\documentclass{jcmlatex}
\usepackage{subfigure}
\usepackage{epstopdf}
\usepackage{algorithm}
\usepackage{xcolor}

\def\JCMvol{xx}
\def\JCMno{x}
\def\JCMyear{200x}
\def\JCMreceived{xxx}
\def\JCMrevised{xxx}
\def\JCMaccepted{xxx}
\def\JCMonline{}
\begin{document}
	
	\markboth{Xie Yixin, Liu Jin-Peng, Sun Cong, Yuan Ya-Xiang}{New gradient methods with 3 dimensional quadratic termination}
	\title{NEW GRADIENT METHODS WITH 3 DIMENSIONAL QUADRATIC TERMINATION
	}
	
	\author{
		Yixin Xie
		\thanks{School of Mathematical Sciences, Beijing University of Posts and Telecommunications (BUPT);\\
		Key Laboratory of Mathematics and Information Networks (BUPT) Ministry of Education,
		Beijing 100876, China			 \\ 
		Email: xieyixin@bupt.edu.cn
	}
	\and
	Jin-Peng Liu
	\thanks{Yau Mathematical Sciences Center, Tsinghua University, Beijing 100084, China;\\
		Yanqi Lake Beijing Institute of Mathematical Sciences and Applications, Beijing 100407, China,	\\ 
		Email: liujinpeng@mail.tsinghua.edu.cn
	}
	\and
	Cong Sun\footnote{Corresponding author}
	\thanks{School of Mathematical Sciences, Beijing University of Posts and Telecommunications;\\
		Key Laboratory of Mathematics and Information Networks (BUPT) Ministry of Education, Beijing 100876, China\\
		Email: suncong86@bupt.edu.cn}
	\and
	Ya-Xiang Yuan
	\thanks{State Key Laboratory of Scientific/Engineering Computing, Institute of Computational Mathematics and Scientific/Engineering Computing, Academy of Mathematics and System Sciences, ChineseAcademy of Sciences, Beijing 100190, China
		\\ Email: yyx@lsec.cc.ac.cn}		
}

	\maketitle
	
	\begin{abstract}
		A new stepsize for gradient method is proposed. Combining it with the exact line search stepsizes, the gradient method achieves the optimal solution in $5$ steps for $3$ dimensional quadratic \nobreak function \nobreak minimization problem. The new stepsize is plugged in the cyclic stepsize update \nobreak strategy, and a new gradient method is proposed. By applying the quadratic interpolation for Cauchy approximation, the proposed gradient method is extended to solve general unconstrained problem. With the improved GLL line search, the global convergence of the \nobreak proposed method is proved. \nobreak Furthermore, its sublinear convergence rate for convex \nobreak problems and $R$-linear convergence rate for problems with quadratic functional growth property are analyzed. Numerical results show that our proposed \nobreak algorithm enjoys good performances in terms of computational cost, and line search requires very few trial \nobreak stepsizes.
	\end{abstract}
	
	

	\section{Introduction}\label{section1}
Gradient method is an important algorithm for large scale optimization \nobreak problems, which appear in many applications such as machine learning \cite{machine2}, weather forecasting \cite{weather2015} and inverse problems \cite{sun2020inverse}. Consider the unconstrained optimization problem:
\begin{equation}\label{f1.1}
	\min_{\mathbf{x}\in\mathbb{R}^{n}}\ f(\mathbf{x}),
\end{equation}
where $f:\mathbb{R}^n\mapsto\mathbb{R}$ is a continuously differentiable and smooth function. The iteration formula of gradient method is:
\begin{equation*}
	\mathbf{x}_{k+1}=\mathbf{x}_{k}-\alpha_{k}\mathbf{g}_{k},
\end{equation*}
where $\mathbf{g}_{k}=\nabla f(\mathbf{x}_{k})$ is the gradient vector at the $k$-th iteration point $\mathbf{x}_k$ and $\alpha_k$ is the stepsize. As in the line search framework, the search direction of gradient method is fixed as the negative gradient, and stepsize is the only tuning parameter. Yet it plays a very important role, which has great influence on the efficiency of gradient method.

Many stepsize strategies have been proposed in the literature. The classic steepest descent (SD) method, proposed by Cauchy\cite{Cauchy1847}, considers the exact line search stepsize:
\begin{equation*}
	\alpha_{k}^{SD}=\mathop{\arg\min}_{\alpha>0}\ f(\mathbf{x}_{k}-\alpha \mathbf{g}_{k}),
\end{equation*}	
which is also called the Cauchy stepsize. For quadratic function
\begin{equation}\label{quadraticf} f(\mathbf{x})=\frac{1}{2}\mathbf{x}^{T}\mathbf{H}\mathbf{x}+\mathbf{b}^{T}\mathbf{x},
\end{equation}
the Cauchy stepsize has the following closed-form solution:
\begin{equation}\label{Cauchy}
	\alpha_{k}^{SD}=\frac{\mathbf{g}_{k}^{T}\mathbf{g}_{k}}{\mathbf{g}_{k}^{T}\mathbf{H}\mathbf{g}_{k}},
\end{equation}		
where $\mathbf{H}\in\mathbb{R}^{n\times n}$ is a symmetric positive definite matrix. It has been proved that the SD method converges $Q$-linearly \cite{Akaike1959}. However, its performance highly depends on the condition number of the problem. The SD method can be very slow, and the zigzag phenomenon is commonly observed when the condition number is large. In recent years, many variations of SD have been proposed \cite{Andreani2020},\cite{di2020}.

In\cite{BB1988}, two stepsizes are proposed by Barzilai and Borwein based on the quasi-Newton property, which are named as BB stepsizes:
\begin{equation}
	\alpha_{k}^{BB1}=\frac{\mathbf{s}_{k-1}^{T}\mathbf{s}_{k-1}}{\mathbf{s}_{k-1}^{T}\mathbf{y}_{k-1}},\quad \alpha_{k}^{BB2}=\frac{\mathbf{s}_{k-1}^{T}\mathbf{y}_{k-1}}{\mathbf{y}_{k-1}^{T}\mathbf{y}_{k-1}},
\end{equation}
where $\mathbf{s}_{k-1}=\mathbf{x}_{k}-\mathbf{x}_{k-1}$ and $\mathbf{y}_{k-1}=\mathbf{g}_{k}-\mathbf{g}_{k-1}$. Remarkably, the BB method is the first gradient method that has superlinear convergence rate for minimizing $2$ dimensional quadratic functions. It has many improved variants. The most widely used and efficient one is ABBmin \cite{ABBmin,ABBmin19}. The stepsize is updated as
\begin{equation*}
	\alpha_{k}^{ABBmin}=
	\left\{
	\begin{aligned}
		\min\{\alpha_{j}^{BB2}|j=\max\{1,k-m\},\cdots,k\}& , & \text{if}\quad \frac{\alpha_{k}^{BB2}}{\alpha_{k}^{BB1}}<\tau_{1};\\
		\alpha_{k}^{BB1}& , & \text{otherwise}
	\end{aligned}
	\right.
\end{equation*}
with $\tau_{1}\in(0,1)$. Essentially, it uses short BB stepsize and switches to the long BB stepsize occasionally. ABBmin is extended in the multipoint stepsize gradient (MPSG) method \cite{Huang2022}, where the difference between two consecutive points is extended to the difference between multiple points.
\begin{equation}\label{MPSG}
	\alpha_{k}^{MPSG}=
	\left\{
	\begin{aligned}
		\min\{\alpha_{k-1}^{MDF2},\alpha_{k}^{MDF2}\}& , && \text{if}\quad  \frac{\alpha_{k}^{MDF2}}{\alpha_{k}^{MDF1}}<\tau_{2};\\
		\alpha_{k}^{MDF1}& , && \text{otherwise},
	\end{aligned}
	\right.
\end{equation}
where $\alpha_k^{MDF1}=\frac{\sum_{i=1}^{m}\mathbf{s}_{k-i}^T\mathbf{s}_{k-i}}{\sum_{i=1}^{m}|\mathbf{s}_{k-i}^T\mathbf{y}_{k-i}|}$, $\alpha_k^{MDF2}=\frac{\sum_{i=1}^{m}|\mathbf{s}_{k-i}^T\mathbf{y}_{k-i}|}{\sum_{i=1}^{m}\mathbf{s}_{k-i}^T\mathbf{s}_{k-i}}$ and $\tau_{2}\in(0,1)$.

In \cite{Yuan2006}, Yuan proposes a new stepsize, which enjoys $3$-step termination for $2$ dimensional quadratic function minimization problems. The iteration formula is as follows:
\begin{equation}\label{2d2p}
	\left\{
	\begin{aligned}
		\mathbf{x}_1&=\mathbf{x}_0-\alpha_{0}^{SD}\mathbf{g}_0,\\
		\mathbf{x}_2&=\mathbf{x}_1-\alpha_1^Y\mathbf{g}_1,\\
		\mathbf{x}^*&=\mathbf{x}_2-\alpha_{2}^{SD}\mathbf{g}_2,
	\end{aligned}
	\right.
\end{equation}
where $\mathbf{x}^*$ is the optimal solution. The Yuan stepsize
\begin{equation}\label{Y}
	\alpha_{k}^{Y}=2\Big(\sqrt{(\frac{1}{\alpha_{k-1}^{SD}}-\frac{1}{\alpha_{k}^{SD}})^{2}+\frac{4\Vert \mathbf{g}_{k}\Vert_{2}^{2}}{\Vert\mathbf{s}_{k-1}\Vert^{2}_{2}}}+\frac{1}{\alpha_{k-1}^{SD}}+\frac{1}{\alpha_{k}^{SD}}\Big)^{-1}
\end{equation}
is carefully designed through coordinate transformation. The following variant form of the Yuan stepsize is shown in \cite{DaiYuan2005}, as for the numerical stability:
\begin{equation}\label{YV}
	\alpha_{k}^{YV}=2\Big(\sqrt{(\frac{1}{\alpha_{k-1}^{SD}}-\frac{1}{\alpha_{k}^{SD}})^{2}+\frac{4\Vert \mathbf{g}_{k}\Vert_{2}^{2}}{(\alpha_{k-1}^{SD}\Vert \mathbf{g}_{k-1}\Vert_{2}^{2})^{2}}}+\frac{1}{\alpha_{k-1}^{SD}}+\frac{1}{\alpha_{k}^{SD}}\Big)^{-1}.
\end{equation}
In the iteration process $(\ref{2d2p})$, it is equivalent as $\alpha_k^Y$, but would be different in other occations.

Motivated by Yuan method, the SDC method combines several Cauchy stepsizes and Yuan stepsize to update stepsize in a cyclic way \cite{Asmundis2014}. In each cycle, it first applies $h(\geq2)$ Cauchy stepsizes, and then uses the Yuan stepsize as the fixed stepsize for $m$ iterations. The stepsize update strategy is
\begin{equation}\label{SDC}
	\alpha_{k}^{SDC}=
	\left\{
	\begin{aligned}
		\alpha_{k}^{SD}& , && \text{if} \mod(k,h+m)<h;\\
		\alpha_{k}^{YV}& , && \text{if} \mod(k,h+m)=h;\\
		\alpha_{k-1}& , && \text{otherwise}.
	\end{aligned}
	\right.
\end{equation}
Another SL framework of the cyclic gradient method is proposed in \cite{Sun2020}, where each loop consists of $2$ Cauchy stepsizes and several gradient stepsizes with fixed stepsize:
\begin{equation}\label{SL}
	\alpha_{k}^{SL}=
	\left\{
	\begin{aligned}
		\alpha_{k}^{SD}& , && \text{if} \mod(k,T)<2;\\
		\alpha_{k-1}^{F}& , && \text{if} \mod(k,T)=2;\\
		\alpha_{k-1}& , && \text{otherwise}.
	\end{aligned}
	\right.
\end{equation}
Here, the authors offer four choices for $\alpha_{k}^{F}$:
\begin{equation*}
	\alpha_{k}^{YV},\quad \tilde{\alpha}_{k}=\Big(\frac{1}{\alpha_{k-1}^{SD}}+\frac{1}{\alpha_{k}^{SD}}\Big)^{-1}, \alpha_{k}^{F\min}=\min\Big\{\alpha_{k-1}^{SD},\alpha_{k}^{SD}\Big\},\quad \alpha_{k}^{F\max}=\max\Big\{\alpha_{k-1}^{SD},\alpha_{k}^{SD}\Big\}.
\end{equation*}
All the candidate fixed stepsizes are computed through the two consecutive Cauchy stepsizes. Note that the fixed stepsize is calculated with the information of the previous ($(k-1)$-th) iteration rather than that of the current ($k$-th) iteration. In this case, only two Cauchy stepsizes are required in each loop. They are computed and used in the first two iterations; then the fixed stepsize is obtained through the already computed two Cauchy stepsizes, and is used for the rest $T-2$ iterations in the loop. Thus its average computational cost is very low. Both SDC and SL methods are designed for problems with quadratic objective functions, and they both have $R$-linear convergence rate. 

Subsequently, the SL method was extended to general unconstrained optimization problems \cite{Zhangya2022, ACINP}. The main difficulty for the extension is that Cauchy stepsize does not have closed form for general problem. In \cite{Zhangya2022}, the BB stepsize in the next iteration is used for the approximation of the Cauchy stepsize in the current iteration, as they are equivalent for the case of quadratic objective function. Another approximation for the Cauchy stepsize is proposed in \cite{ACINP}, which makes use of the quadratic interpolation idea. The cyclic stepsize update strategies perform well numerically and the corresponding gradient methods enjoy low computational cost per iteration. 

Clearly the Yuan stepsize is an efficient stepsize which enjoys finite termination for $2$ dimensional quadratic function minimization. The aforementioned cyclic gradient methods are based or partly based on it. In this work a new stepsize for the gradient method is proposed, which has finite termination property for $3$ dimensional quadratic function minimization problems. It can be treated as the extension of Yuan stepsize, and is degenerated into Yuan stepsize in the $2$ dimensional subspace. In Section \ref{section2}, a new stepsize is derived which ensures $5$-step termination for $3$ dimensional quadratic minimization problems. The proposed stepsize is used in the cyclic gradient method framework, and the method is generalized to general problem in Section \ref{section3}. The theoretical analysis for the proposed gradient method is established, including the global convergence and convergence rates for different types of functions. Numerical experiments comparing the proposed method with other state-of-the-art algorithms are presented in Section \ref{section4}. Finally, the paper is concluded in Section \ref{section5}.
	
	\section{New Stepsize with $3$ Dimensional $5$-step Termination}\label{section2}
	
	In this section, a new stepsize is proposed for gradient method, with the motivation of finite termination of $3$ dimensional quadratic function problems.
	
	Recall that Yuan stepsize $(\ref{Y})$ enjoys $3$-step termination for $2$ dimensional quadratic function \nobreak problems. Now we consider the following $3$ dimensional problem:
	\begin{equation}\label{3dimfb0}
		\min_{\mathbf{x}\in\mathbb{R}^{3}}f(\mathbf{x})=\frac
		{1}{2}\mathbf{x}^T\mathbf{H}\mathbf{x}+\mathbf{b}^T\mathbf{x},
	\end{equation}
	where the Hessian matrix $\mathbf{H}\in\mathbb{R}^{3\times3}$ is symmetric and positive definite. We would like to use gradient method to solve it in $5$ steps. The following iteration process is considered:
	\begin{equation}\label{2dim5}
		\left\{
		\begin{aligned}
			\mathbf{x}_{1}&=\mathbf{x}_{0}-\alpha_{0}^{SD}\mathbf{g}_{0},\\
			\mathbf{x}_{2}&=\mathbf{x}_{1}-\alpha_{1}^{SD}\mathbf{g}_{1},\\
			\mathbf{x}_{3}&=\mathbf{x}_{2}-\alpha_{2}\mathbf{g}_{2},\\
			\mathbf{x}_{4}&=\mathbf{x}_{3}-\alpha_{3}\mathbf{g}_{3},\\
			\mathbf{x}^*&=\mathbf{x}_{4}-\alpha_{4}^{SD}\mathbf{g}_{4},
		\end{aligned}
		\right.
	\end{equation}
	where $\mathbf{x}^*$ is the optimal solution of $(\ref{3dimfb0})$. In our design, the first two iterations apply the Cauchy stepsizes:
	\begin{equation*}
		\alpha_0^{SD}=\frac{\mathbf{g}_0^T\mathbf{g}_0}{\mathbf{g}_0^T\mathbf{H}\mathbf{g}_0},\quad \alpha_1^{SD}=\frac{\mathbf{g}_1^T\mathbf{g}_1}{\mathbf{g}_1^T\mathbf{H}\mathbf{g}_1}.
	\end{equation*}
 After that the two stepsizes $\alpha_2$ and $\alpha_3$ require specific design. For the last iteration the Cauchy stepsize is used again.
	
	Since the first two steps are exact line search steps, we have that $\mathbf{g}_0\perp \mathbf{g}_{1}$ and $\mathbf{g}_{1}\perp \mathbf{g}_{2}$. Suppose $\mathbf{g}_{0}$ and $\mathbf{g}_{2}$ are not parallel. Based on $\mathbf{g}_{0}$, $\mathbf{g}_{1}$ and $\mathbf{g}_{2}$, a new rectangular coordinate system can be built up. Applying the Schmidt orthogonalization, we generate $\tilde{\mathbf{g}}_{2}$ to be mutually orthogonal with $\mathbf{g}_0$ and $\mathbf{g}_1$:
	\begin{equation}
		\tilde{\mathbf{g}}_{2}=\mathbf{g}_{2}-\frac{\mathbf{g}_{2}^{T}\mathbf{g}_{0}}{\Vert\mathbf{g}_{0}\Vert^{2}}\mathbf{g}_{0}.
	\end{equation}
	For any $\mathbf{x}\in\mathbb{R}^3$, it can be expressed as
	\begin{equation*}
		\mathbf{x}=\mathbf{x}_2+u\frac{\mathbf{g}_{0}}{\Vert\mathbf{g}_{0}\Vert}+v\frac{\mathbf{g}_{1}}{\Vert\mathbf{g}_{1}\Vert}+w\frac{\tilde{\mathbf{g}}_{2}}{\Vert\tilde{\mathbf{g}}_{2}\Vert}=\mathbf{x}_2+\mathbf{Q}\mathbf{y},
	\end{equation*}
where $\mathbf{Q}=\Big(\frac{\mathbf{g}_{0}}{\Vert\mathbf{g}_{0}\Vert}, \frac{\mathbf{g}_{1}}{\Vert\mathbf{g}_{1}\Vert},\frac{\tilde{\mathbf{g}}_{2}}{\Vert\tilde{\mathbf{g}}_{2}\Vert} \Big)$, and $\mathbf{y}=(u,v,w)^{T}$ is the new coordinate for $\mathbf{x}$ in the new coordinate system. The objective function is equivalent to
	\begin{equation}
		\begin{aligned}
			f(\mathbf{x})&=f(\mathbf{x}_2+\mathbf{Q}\mathbf{y})=\frac{1}{2}(\mathbf{x}_2+\mathbf{Q}\mathbf{y})^T\mathbf{H}(\mathbf{x}_2+\mathbf{Q}\mathbf{y})+\mathbf{b}^T(\mathbf{x}_2+\mathbf{Q}\mathbf{y})\\
			&=f(\mathbf{x}_2)+\mathbf{a}^T\mathbf{y}+\frac{1}{2}\mathbf{y}^T\mathbf{A}\mathbf{y}\\
			&\triangleq h(\mathbf{y}).
		\end{aligned}
	\end{equation}
	Here $\mathbf{a}=\mathbf{Q}^T\mathbf{g}_2=\Big(\frac{\mathbf{g}_2^T\mathbf{g}_0}{\Vert\mathbf{g}_0\Vert},0,\Vert\tilde{\mathbf{g}}_2\Vert\Big)^T$, and
	\begin{equation}\label{Q_a33}
		\mathbf{A}=\mathbf{Q}^T\mathbf{H}\mathbf{Q}\triangleq (a_{ij})_{3\times 3}=
		\begin{pmatrix}
			\frac{1}{\alpha_0^{SD}} & -\sqrt{\beta\gamma} & 0\\
			-\sqrt{\beta\gamma}  & \frac{1}{\alpha_1^{SD}} & -\sqrt{\beta(1-\gamma)}\\
			0 & -\sqrt{\beta(1-\gamma)} & a_{33}
		\end{pmatrix},
	\end{equation}
where
\begin{equation*}
		\beta=\frac{\Vert\mathbf{g}_2\Vert^2}{(\alpha_1^{SD})^2\Vert\mathbf{g}_1\Vert^2},\quad
		\gamma=\frac{(\mathbf{g}_2^T\mathbf{g}_0)^2}{\Vert\mathbf{g}_0\Vert^2\Vert\mathbf{g}_2\Vert^2},\quad a_{33}=\frac{\tilde{\mathbf{g}}_{2}\mathbf{H}\tilde{\mathbf{g}}_{2}}{\Vert\tilde{\mathbf{g}}_{2}\Vert^2}=\frac{\frac{1}{\alpha_2^{SD}}-\frac{1}{\alpha_0^{SD}}\gamma}{1-\gamma}.
\end{equation*}
	The iteration points and their gradients can be expressed by the new coordinates in the new coordinate system. Table \ref{table1} lists the $3$ dimensional vectors and their new coordinates.
	\begin{table}[h!]
		\begin{center}
			\caption{old and new coordinates.} \label{table1}
			\begin{tabular}{|c|c|}
				\hline
				\textbf{original coordinates ($\mathbf{x}$)} & \textbf{new coordinates ($\mathbf{y}$)} \\
				\hline
				$\mathbf{x}^*$ & $\mathbf{y}^*=-\mathbf{A}^{-1}\mathbf{a}$ \\
				\hline
				$\mathbf{x}_3=\mathbf{x}_2-\alpha_2\mathbf{g}_2$ & $\mathbf{y}_3=-\alpha_2\mathbf{Q}^T\mathbf{g}_2=-\alpha_2\mathbf{a}$ \\
				$\mathbf{g}_3=\nabla f(\mathbf{x}_3)$ & $\mathbf{g}_{\mathbf{y}_3}=\nabla h(\mathbf{y}_3)=\mathbf{A}\mathbf{y}_3+\mathbf{a}=(\mathbf{I}-\alpha_2\mathbf{A})\mathbf{a}$\\
				\hline
				$\mathbf{x}_4=\mathbf{x}_3-\alpha_3\mathbf{g}_3$ & $\mathbf{y}_4=[\alpha_2\alpha_3\mathbf{A}-(\alpha_2+\alpha_3)\mathbf{I}]\mathbf{a}$ \\
				$\mathbf{g}_4=\nabla f(\mathbf{x}_4)$ & $\mathbf{g}_{\mathbf{y}_4}=\nabla h(\mathbf{y}_4)=[\alpha_2\alpha_3\mathbf{A}-(\alpha_2+\alpha_3)\mathbf{I}+\mathbf{I}]\mathbf{a}$\\
				\hline
			\end{tabular}
		\end{center}
	\end{table}

To satisfy the $5$ step termination, we need the gradient direction $\mathbf{g}_{\mathbf{y}_k}$ to be parallel to the residual vector $\mathbf{y}^*-\mathbf{y}_4$. With the new coordinates, we have
\begin{equation*}
	\mathbf{y}^*-\mathbf{y}_4=-\alpha_4^{SD}\mathbf{g}_{\mathbf{y}_4}.
\end{equation*}
Plugging in the expressions of the new coordinates, we further have
\begin{equation*}
	-\mathbf{A}^{-1}\mathbf{a}-[\alpha_2\alpha_3\mathbf{A}-(\alpha_2+\alpha_3)\mathbf{I}]\mathbf{a}=-\alpha_4^{SD}[\alpha_2\alpha_3\mathbf{A}^2-(\alpha_2+\alpha_3)\mathbf{A}+\mathbf{I}]\mathbf{a}.
\end{equation*}
Through equivalent transformation, we can deduce that
\begin{equation}\label{eq*}
	(\mathbf{I} - \alpha_2 \mathbf{A})(\mathbf{I} - \alpha_3 \mathbf{A})(\mathbf{I} - \alpha_4^{SD} \mathbf{A})\mathbf{a}= \mathbf{0}.
\end{equation}
Since $\mathbf{A}=\mathbf{Q}^T\mathbf{H}\mathbf{Q}$ is symmetric positive definite, it can be diagonalized as
\begin{equation}\label{3dq2}
	\mathbf{A} = \mathbf{U}\Lambda\mathbf{U}^T.
\end{equation}
here $\mathbf{U}$ is an orthogonal matrix and $\boldsymbol{\Lambda} = \text{Diag}(\mu_1, \mu_2, \mu_3)$, where the diagonal elements are the eigenvalues of $\mathbf{A}$ with $\mu_1\geq \mu_2\geq \mu_3>0$. Taking (\ref{3dq2}) into (\ref{eq*}), we have the following nonlinear equations with respect to $\alpha_2, \alpha_3$ and $\alpha_4^{SD}$:
\begin{equation}\label{al_lam}
	\left\{
	\begin{aligned}
		(\alpha_2\mu_1-1)(\alpha_3\mu_1-1)(\alpha_4^{SD}\mu_1-1)&=0;\\
		(\alpha_2\mu_2-1)(\alpha_3\mu_2-1)(\alpha_4^{SD}\mu_2-1)&=0;\\
		(\alpha_2\mu_3-1)(\alpha_3\mu_3-1)(\alpha_4^{SD}\mu_3-1)&=0.
	\end{aligned}
	\right.
\end{equation}
It is observed that as long as $\alpha_2, \alpha_3$ and $\alpha_4^{SD}$ are the inverse of the eigenvalues of $\mathbf{A}$, the equations (\ref{al_lam}) hold. So next we will calculate the eigenvalues of $\mathbf{A}$, and use their inverse to design new stepsizes. Through the eigenvalue equation of $\mathbf{A}$:
\begin{equation}
	|\mu \mathbf{I} - \mathbf{A}| = 0,
\end{equation}
	we have the following cubic equation:
	\begin{equation}\label{2.9}
		\mu^3-t_1\mu^2+t_2\mu-t_3=0,
	\end{equation}
	where
	\begin{equation}\label{t}
		\left\{
		\begin{aligned}
			t_1&=\frac{1}{\alpha_0^{SD}}+\frac{1}{\alpha_1^{SD}}+a_{33};\\
			t_2&=\frac{1}{\alpha_0^{SD}\alpha_1^{SD}}+(\frac{1}{\alpha_0^{SD}}+\frac{1}{\alpha_1^{SD}})a_{33}-\beta;\\
			t_3&=\frac{1}{\alpha_0^{SD}\alpha_1^{SD}}a_{33}-\frac{1}{\alpha_0^{SD}}\beta(1-\gamma)-a_{33}\beta\gamma.
		\end{aligned}
		\right.
	\end{equation}
	According to Cardano's formula, the three roots are
	\begin{equation}\label{mu}
		\begin{split}
			\mu_{1}&=\frac{t_{1}}{3}+2\sqrt{-\frac{p}{3}}\cos\Big(\frac{1}{3}\arccos\big(\frac{3q}{2p}\sqrt{\frac{-3}{p}}\big)\Big),\\
			\mu_{2}&=\frac{t_{1}}{3}+2\sqrt{-\frac{p}{3}}\cos\Big(\frac{1}{3}\arccos\big(\frac{3q}{2p}\sqrt{\frac{-3}{p}}\big)-\frac{2\pi}{3}\Big),\\
			\mu_{3}&=\frac{t_{1}}{3}+2\sqrt{-\frac{p}{3}}\cos\Big(\frac{1}{3}\arccos\big(\frac{3q}{2p}\sqrt{\frac{-3}{p}}\big)+\frac{2\pi}{3}\Big),
		\end{split}
	\end{equation}
	where $p=t_2-\frac{t_1^2}{3}$ and $q=-\frac{2t_1^3}{27}+\frac{t_1t_2}{3}-t_3$. Thus we obtain the three eigenvalues of A as in (\ref{mu}), and three new stepsizes as
	\begin{equation}\label{NYsteps}
		\alpha^{NY(i)}=\mu_{i}^{-1},\quad i=1,2,3.
	\end{equation}
Their relationship is shown in the following theorem.

\begin{theorem} \label{thm:stepsize_properties}
	The stepsizes $\alpha^{NY(i)}$ ($i=1,2,3$) defined in $(\ref{NYsteps})$ satisfy the following order:
	\begin{equation}
		\alpha^{NY(1)} \leq \alpha^{NY(2)} \leq \alpha^{NY(3)}.
	\end{equation}
\end{theorem}

\begin{proof}
	Since $\arccos\left(\frac{3q}{2p}\sqrt{-\frac{3}{p}}\right)$ belongs to the set $[0,\pi]$, it is easy to see that
	\begin{equation*}
		\phi = \frac{1}{3}\arccos\left(\frac{3q}{2p}\sqrt{-\frac{3}{p}}\right)\in [0, \frac{\pi}{3}].
	\end{equation*}
	 Consequently, the three phase angles involved in $(\ref{mu})$ satisfy:
	\begin{equation*}
		\phi \in \left[0, \frac{\pi}{3}\right], \quad 
		\phi - \frac{2\pi}{3} \in \left[-\frac{2\pi}{3}, -\frac{\pi}{3}\right], \quad 
		\phi + \frac{2\pi}{3} \in \left[\frac{2\pi}{3}, \pi\right].
	\end{equation*}
	
	Based on the monotonicity of the cosine function in these intervals, we have:
	\begin{equation*}
		-1 \leq \cos\left(\phi + \frac{2\pi}{3}\right) \leq -\frac{1}{2}\leq \cos\left(\phi - \frac{2\pi}{3}\right)\leq \frac{1}{2}\leq\cos\phi \leq 1.
	\end{equation*}
	Thus we can deduce that 
	\begin{equation*}
		\frac{t_1}{3} - 2\sqrt{-\frac{p}{3}}\leq\mu_3 \leq\frac{t_1}{3} - \sqrt{-\frac{p}{3}}\leq \mu_2 \leq\frac{t_1}{3} + \sqrt{-\frac{p}{3}}\leq \mu_1\leq\frac{t_1}{3} + 2\sqrt{-\frac{p}{3}}.
	\end{equation*}
	For the stepsizes which are the inverse of the eigenvalues, their orders are 
	\begin{equation*}
		\alpha^{NY(1)} \leq \alpha^{NY(2)} \leq \alpha^{NY(3)}.\tag*{$\square$}
	\end{equation*}
\end{proof}

When we build up the new coordinate system, it is assumed that  $\mathbf{g}_0$, $\mathbf{g}_1$ and $\mathbf{g}_2$ are linearly independent. That is, $\mathbf{g}_0$ and $\mathbf{g}_2$ are supposed to be nonparallel. The next theorem reveals that the proposed stepsize is indeed an extension of Yuan stepsize (\ref{Y}), and degenerates into Yuan stepsize in $2$ dimensional space.

\begin{theorem} \label{thm:geometric_properties}
	If $\mathbf{g}_0$ is parallel to $\mathbf{g}_2$, the proposed stepsize $\alpha^{NY(1)}$ degenerates to the Yuan stepsize $\alpha^{Y}$.
\end{theorem}

\begin{proof}
	Yuan stepsize (\ref{Y}) is introduced in \cite{Yuan2006}. Its idea is to build up $2$ dimensional coordinate system and design Yuan stepsize with $2$ dimensional quadratic termination. For our analysis, if $\mathbf{g}_0$ is parallel to $\mathbf{g}_2$, then it degenerates into a $2$ dimensional subspace. The parameter
	\begin{equation*}
		\gamma = \frac{(\mathbf{g}_2^T \mathbf{g}_0)^2}{\Vert\mathbf{g}_0\Vert^2 \|\mathbf{g}_2\|^2} = 1
	\end{equation*}
	holds. Take it into (\ref{2.9}), the equation is degenerated into 
	\begin{equation*}
		\left(\mu - \frac{1}{\alpha_0^{SD}}\right)\left(\mu - \frac{1}{\alpha_1^{SD}}\right) - \beta= 0.
	\end{equation*}
	The above equation is exactly the same as equation in \cite[(2.15)]{Yuan2006}. Since $\alpha^{NY(1)}$ is the shortest stepsize among the three new stepsize, in the $2$ dimensional subspace it degenerates into $\alpha^Y$, which is the shorter stepsize for $2$ dimensional case.\hfill $\square$
	
\end{proof}

	By utilizing the derived stepsizes $\alpha^{NY(i)},i=1,2,3$, we are able to find the optimal solution in $5$ iterations for $3$ dimensional quadratic function minimization problem. Next we will use the shortest stepsize $\alpha^{NY}=\alpha^{NY(1)}$ to design efficient gradient methods. Since it is the inverse eigenvalue of the $3$ dimensional Hessian matrix, with gradient step and stepsize $\alpha^{NY}$, the iteration point would fall into $2$ dimensional subspace from $3$ dimensional problem. In the $2$ dimensional subspace $\alpha^{NY}$ becomes $\alpha^{Y}$, which again reduces the subspace dimension. In this way, when the iteration point falls into $3$ dimensional subspace, the gradient method with the new stepsize and Cauchy stepsize would terminate immediately.

\section{New Gradient Method}\label{section3}

In this section, the proposed stepsize $\alpha^{NY}$ is used in the new gradient method. As mentioned in Section \ref{section1}, cyclic stepsize update strategy (\ref{SL}) is proper for large scale problem. With several repeated stepsizes in one cycle, the average computational cost in each iteration is reduced. Thus here we plug the stepsize $\alpha^{NY}$ into the cyclic gradient method framework to generate new gradient method. The discussions are first made for the quadratic function minimization problem, and then the gradient method is generalized to the general case.

\subsection{Quadratic case}

Consider the problem with quadratic objective function (\ref{3dimfb0}). Combining with the cyclic framework proposed in \cite{Sun2020}, our new stepsize update strategy is
\begin{equation}\label{NY}
	\alpha_{k}=
	\left\{
	\begin{aligned}
		\alpha_{k}^{SD},\quad &\mod(k,T)<2;\\
		\alpha_{k}^{NY},\quad &\mod(k,T)=2;\\
		\alpha_{k-1},\quad &\text{otherwise},
	\end{aligned}
	\right.
\end{equation}
where $T$ is the number of iterations in one cycle. Different from (\ref{SL}), the fixed stepsize $\alpha_k^{NY}$ is computed via the information of the current (k-th) iteration. Three Cauchy stepsizes are computed in one cycle. Using the new stepsize update strategy (\ref{NY}), we summarize the corresponding gradient method in Algorithm \ref{algNY}, which is named as NY. 

\begin{algorithm}[h]
	\caption{NY algorithm (quadratic case)}
	\label{algNY}
	\textbf{Input:} Initial point $\mathbf{x}_0$; $k=0$; the iteration number in each cycle $T$; the stopping parameters $\epsilon$. \\
	\textbf{Output:} $\mathbf{x}_k$.

	\begin{tabbing}
		\hspace{1.5em} \= \hspace{1.5em} \= \hspace{1.5em} \= \kill
		
		\textbf{while} $\|\mathbf{g}_k\| > \epsilon \|\mathbf{g}_0\|$ \textbf{do} \\
		
		\> \textbf{if} $\text{mod}(k, T) < 2$ \textbf{then} \\
		\> \> Let $\alpha_k = \alpha_k^{SD}$; \\
		
		\> \textbf{else if} $\text{mod}(k, T) = 2$ \textbf{then} \\
		\> \> Let $\alpha_k = \alpha_k^{NY}$; \\
		
		\> \textbf{else} \\
		\> \> Let $\alpha_k = \alpha_{k-1}$; \\
		
		\> Let $\mathbf{x}_{k+1} = \mathbf{x}_k - \alpha_k \mathbf{g}_k$ and $k := k + 1$. \\
		
		\textbf{end while}
	\end{tabbing}
\end{algorithm}
As analyzed in Section \ref{section2}, the proposed stepsize has the property of $3$ dimensional quadratic termination. The following theorem verifies the conclusion.

\begin{theorem}\label{3d5f}
	For $3$ dimensional convex quadratic function minimization problem, Algorithm $\ref{algNY}$ \nobreak terminates in $2T+1$ iterations.
\end{theorem}

\begin{proof}
	Without loss of generality, suppose the Hessian matrix in $(\ref{3dimfb0})$ is 
	\begin{equation*}
		\boldsymbol{H} = \text{Diag}(\lambda_1, \lambda_2, \lambda_3)
	\end{equation*}
	with $\lambda_1 \geq \lambda_2 \geq \lambda_3 > 0$, and $\boldsymbol{b}=\boldsymbol{0}$. Let $\boldsymbol{g}_k = (\boldsymbol{g}_k^{(1)}, \boldsymbol{g}_k^{(2)}, \boldsymbol{g}_k^{(3)})^T$. The iteration yields
	\begin{equation}\label{g_kgk}
		\boldsymbol{g}_{k+1} = (\boldsymbol{I} - \alpha_k \boldsymbol{H})\boldsymbol{g}_k,
	\end{equation}
	 implying
	\begin{equation}\label{g_fl}
		\boldsymbol{g}_{k+1}^{(i)} = (1 - \lambda_i \alpha_k)\boldsymbol{g}_k^{(i)}, \quad i=1,2,3.
	\end{equation}
	
	In the first cycle ($0 \le k < T$), at step $k=2$, we have
	\[
	\alpha_2 = \alpha_2^{NY} = \frac{1}{\lambda_1}.
	\]
	Through $(\ref{g_fl})$, we have
	\[
	\boldsymbol{g}_{3}^{(1)} = (1 - \lambda_1 \lambda_1^{-1})\boldsymbol{g}_{2}^{(1)} = 0,
	\]
	which implies $\boldsymbol{g}_k^{(1)} = 0$ for all $k \geq 3$.
	
	In the second cycle ($T \le k < 2T$), at step $k=T+2$, we should have $\alpha_{T+2}=\alpha_{T+2}^{NY}$. Since $\boldsymbol{g}_{T+2}^{(1)} = 0$, problem (\ref{3dimfb0}) degenerates into a $2$ dimensional problem. According to Theorem \ref{thm:geometric_properties}, in this case we have
	\[
	\alpha_{T+2} =  \alpha^{YV}_{T+2} =\frac{1}{ \lambda_2}.
	\]
	This leads to
	\[
	\boldsymbol{g}_{T+3}^{(2)} = (1 - \lambda_2 \frac{1}{\lambda_2})\boldsymbol{g}_{T+2}^{(2)} = 0.
	\]
	And for all $k\geq T+3$, it holds that $\boldsymbol{g}_k^{(1)}=\boldsymbol{g}_k^{(2)}=0$. The problem now becomes a $1$ dimensional problem.
	
	In the first step of the third cycle $(k=2T+1)$, the Cauchy stepsize is applied. Since the problem now is $1$ dimensional, it goes to the optimal solution after the exact line search. That is $\boldsymbol{g}_{2T+1}^{(3)}=0$ and thus we have $\boldsymbol{g}_{2T+1}=\mathbf{0}$.
	
	From the above analysis, we can conclude that Algorithm $\ref{algNY}$ terminates in $2T+1$ iterations.
	
	\hfill $\square$
\end{proof}

Theorem $\ref{3d5f}$ guarantees the finite termination of the proposed NY method for $3$ dimensional \nobreak problem. Next, the convergence result of NY for general $n$ dimensional problem is analyzed. 
\begin{lemma}\label{lemma3.1}
	For $n$ dimensional problem (\ref{3dimfb0}), assume that $\mathbf{H}=Diag(\lambda_1, \lambda_2, \cdots, \lambda_n)$ is a diagonal matrix, where $\lambda_1\geq\lambda_2\geq\cdots\geq\lambda_n=1$ and $\{\mathbf{d}_1, \mathbf{d}_2, \cdots, \mathbf{d}_n\}$ is a set of associated orthonormal eigenvectors. Algorithm \ref{algNY} is applied to problem $(\ref{3dimfb0})$. Suppose that the initial point $\mathbf{x}_0$ satisfies $\mathbf{g}_0^T\mathbf{d}_1\neq0$ and $\mathbf{g}_0^T\mathbf{d}_n\neq0$. Let $\mathbf{g}_k=\sum_{i=1}^{n}\lambda_i^k\mathbf{d}_i$ and define $h(k,l)=\sum_{i=l}^{n}(\lambda_i^k)^2$. Then Algorithm \ref{algNY} satisfies the following two properties.\\
	(i) There exists a constant $M_1\geq\lambda_n$, such that $\lambda_n\leq\alpha_k^{-1}\leq M_1$ holds for all $k>0$.\\
	(ii) If there exists a positive integer $k_0$ and a positive constant $M_2$, such that for any $l\in\{2,...,n\}$ and $\forall\epsilon>0$, $h(k-j,l)<\epsilon$ and $(\lambda_{l-1}^{k-j})^2\geq M_2\epsilon$ hold for any $j\in\{0,\cdots,\min\{k,k_0\}\}$, then $\alpha_k^{-1}\geq M_3\lambda_{l-1}$, where $\frac{1}{2}<M_3<1$ is a constant.
\end{lemma}
\begin{proof}
	According to the closed form of (\ref{Cauchy}), it is trivial to see that
	\begin{equation}\label{lamb_sd}
		\lambda_n\leq(\alpha_k^{SD})^{-1}\leq\lambda_1
	\end{equation}
	 holds for any $k$. Thus from the expression of $a_{ii}, i=1,2,3$ in (\ref{Q_a33}), we have $\lambda_n\leq a_{ii}\leq\lambda_1$ for $i=1,2,3$.
	
	According to the analysis in Theorem \ref{thm:stepsize_properties} and (\ref{t}), we can deduce that
	\begin{equation*}
		(\alpha_{k}^{NY})^{-1}\geq\frac{t_1}{3}+\sqrt{-\frac{p}{3}}\geq\frac{t_1}{3}\geq\lambda_n.
	\end{equation*}
	 Furthermore, 
	\begin{equation*}
		\begin{split}
			-\frac{p}{3}&=\frac{\beta}{3}+\frac{1}{9}(a_{11}+a_{22}+a_{33})^2-\frac{1}{3}(a_{11}a_{22}+a_{22}a_{33}+a_{11}a_{33})\\
			&=\frac{\beta}{3}+\frac{1}{18}[(a_{11}-a_{22})^2+(a_{22}-a_{33})^2+(a_{11}-a_{33})^2]\\
			&\leq\frac{\lambda_1^2-\lambda_n^2}{3}+\frac{(\lambda_1-\lambda_n)^2}{6}=\frac{(\lambda_1-\lambda_n)(3\lambda_1+\lambda_n)}{6}.
		\end{split}
	\end{equation*}
The inequality holds due to (\ref{g_kgk}) and
\begin{equation*}
	\beta=\frac{\Vert\mathbf{g}_k\Vert^2}{(\alpha^{SD}_{k-1})^2\Vert\mathbf{g}_{k-1}\Vert^2}=\frac{\mathbf{g}_{k-1}^T\mathbf{H}^2\mathbf{g}_{k-1}}{\mathbf{g}_{k-1}^T\mathbf{g}_{k-1}}-(\frac{\mathbf{g}_{k-1}^T\mathbf{H}\mathbf{g}_{k-1}}{\mathbf{g}_{k-1}^T\mathbf{g}_{k-1}})^2\leq\lambda_1^2-\lambda_n^2.
\end{equation*}
	
	Let $M_1=\lambda_1+\sqrt{\frac{(\lambda_1-\lambda_n)(3\lambda_1+\lambda_n)}{6}}$, we have
	\begin{equation*}
		(\alpha_{k}^{NY})^{-1}\leq\frac{t_1}{3}+2\sqrt{-\frac{p}{3}}\leq M_1.
	\end{equation*}
	 From (\ref{lamb_sd}) we can see that $\lambda_n\leq(\alpha_k^{NY})^{-1}\leq M_1$ for all $k$. This shows that Algorithm \ref{algNY} satisfies Property $(i)$.
	
	From the deduction in \cite[(4.3)]{DaiAS}, it holds that $(\alpha_k^{SD})^{-1}\geq\frac{2}{3}\lambda_{l-1}$ for all $k$. Next, we shall discuss Property (ii) from two situations. Let $M_2=2$ and $M_3=\frac{4+\sqrt{2}}{9}$. \\
	(1) If $(\alpha^{SD}_k)^{-1}\geq(\alpha^{SD}_{k-2})^{-1}$, we have that 
	\begin{equation}\label{a33}
		a_{33}=\frac{\frac{1}{\alpha_k^{SD}}-\frac{1}{\alpha^{SD}_{k-2}}\gamma}{1-\gamma}\geq(\alpha^{SD}_{k-2})^{-1}.
	\end{equation}
	Thus 
	\begin{equation}
		(\alpha_k^{NY})^{-1}\geq\frac{t_1}{3}+\sqrt{-\frac{p}{3}}\geq\frac{t_1}{3}\geq\frac{2(\alpha^{SD}_{k-2})^{-1}+(\alpha^{SD}_{k-1})^{-1}}{3}\geq\frac{2}{3}\lambda_{l-1}\geq M_3\lambda_{l-1}.
	\end{equation}
	\\
	(2) If $(\alpha^{SD}_k)^{-1}<(\alpha^{SD}_{k-2})^{-1}$, the inequality $(\ref{a33})$ does not hold. Let 
	\begin{equation*}
		(\alpha^{SD}_k)^{-1}=t(\alpha^{SD}_{k-2})^{-1},\quad 0<t<1.
	\end{equation*}
	Since $a_{33}=\frac{\tilde{\mathbf{g}}_{k}\mathbf{H}\tilde{\mathbf{g}}_{k}}{\Vert\tilde{\mathbf{g}}_{k}\Vert^{2}}>0$, we have $\gamma<t<1$. Then we can deduce that
	\begin{equation*}
		\begin{split}
			a_{33}&=\frac{t-\gamma}{1-\gamma}\cdotp(\alpha^{SD}_{k-2})^{-1}\leq (\alpha_{k-2}^{SD})^{-1},\\
			\frac{t_1}{3}&=\frac{a_{11}+a_{22}+a_{33}}{3}\geq(\frac{4}{9}+\frac{2}{9}\cdotp\frac{t-\gamma}{1-\gamma})\lambda_{l-1}.
		\end{split}
	\end{equation*}
	Consequently,
	\begin{equation*}
		\begin{aligned}
			\sqrt{-\frac{p}{3}}&=\sqrt{\frac{\beta}{3}+\frac{1}{18}[(a_{11}-a_{22})^2+(a_{22}-a_{33})^2+(a_{11}-a_{33})^2]}\\
			&\geq\sqrt{\frac{(a_{11}-a_{33})^2}{18}}=\frac{1}{3\sqrt{2}}\cdotp(1-\frac{t-\gamma}{1-\gamma})(\alpha^{SD}_{k-2})^{-1}\geq\frac{\sqrt{2}}{9}(1-\frac{t-\gamma}{1-\gamma})\lambda_{l-1},
		\end{aligned}
	\end{equation*}
	and 
	\begin{equation*}
		(\alpha_k^{NY})^{-1}\geq\frac{t_1}{3}+\sqrt{-\frac{p}{3}}\geq(\frac{4}{9}+\frac{\sqrt{2}}{9}+\frac{2-\sqrt{2}}{9}\cdotp\frac{t-\gamma}{1-\gamma})\lambda_{l-1}\geq M_3\lambda_{l-1}.
	\end{equation*}
	Summarizing the two situations, Property (ii) holds.\hfill $\square$
\end{proof}

\begin{theorem}\label{thm:linear_convergence}
	Under the same assumptions as in Lemma \ref{lemma3.1}, Algorithm $\ref{algNY}$ either terminates in finite iterations or the sequence $\{\|\mathbf{g}_k\|\}$ converges to zero R-linearly.
\end{theorem}

\begin{proof}
	We can get to the conclusion following the proof of\cite[Theorem 4.1]{DaiAS}. The key idea in its proof is that $M_3=\frac{2}{3}$, and $\delta_1=\max\{(1-\frac{\lambda_n}{M_1})^2, (1-\frac{1}{M_3})^2\}<1$ \footnote{Here we use the symbols in our paper, where $\lambda_n$ represents the smallest eigenvalue of the Hessian matrix $
		\boldsymbol{H}$.}. This guarantees the inequalities in\cite[Theorem 4.1]{DaiAS} hold, and further achieves the result that $\Vert \boldsymbol{g}_{k+M}\Vert^2\leq\frac{1}{4}\Vert \boldsymbol{g}_{k}\Vert^2$. 
	
	Here in our work, we have $M_3=\frac{4+\sqrt{2}}{9}$, and it still holds that
	\begin{equation*}
		\delta_1=\max\{(1-\frac{\lambda_n}{M_1})^2, (1-\frac{1}{M_3})^2\}<1.
	\end{equation*}
	Consequently we can get the same conclusion that $\Vert \boldsymbol{g}_{k+M}\Vert^2\leq\frac{1}{4}\Vert \boldsymbol{g}_{k}\Vert^2$. Thus Theorem 3.2 holds.\hfill $\square$
\end{proof}

\subsection{General case}
In this subsection, the proposed NY method is extended to solve the general unconstrained problem (\ref{f1.1}).

Since there is no special structure for the objective function $f(\mathbf{x})$, the main difficulty for the extension of the NY algorithm is the construction of Cauchy stepsize. It has no closed form solution. Here we adopt the quadratic interpolation idea \cite{ACINP} to approximate Cauchy stepsizes. Mainly, in the k-th iteration, we use the function value and gradient of the current iteration point $\mathbf{x}_k$ and the function value of a new point $\hat{\mathbf{x}}=\mathbf{x}_k-\beta_0\mathbf{g}_k$ to build up a $1$ dimensional quadratic interpolation function for $ \phi(\alpha)\triangleq f(\mathbf{x}_k-\alpha \mathbf{g}_k)$. If it does not work well, $\hat{\mathbf{x}}$ would be replaced by the minimal point of the last quadratic interpolation function. The generation process for the approximated Cauchy stepsize $\alpha_k^{ASD}$ is shown in \cite[Algorithm 1]{ACINP}. By using the approximation of Cauchy stepsize in \cite{ACINP}, we can extend our Algorithm \ref{algNY} to solving problem (\ref{f1.1}). In the new algorithm, all Cauchy stepsizes $\alpha^{SD}$ are replaced by the approximated Cauchy stepsizes $\alpha_k^{ASD}$. Thus the proposed stepsize $\alpha_k^{NY}$ is also approximated, named as $\alpha_k^{ANY}$. The new stepsize updated strategy is

\begin{equation}\label{alANYk}
	\alpha_{k}=
	\left\{
	\begin{aligned}
		\alpha_{k}^{ASD},\quad &\mod(k,T)<2;\\
		\alpha_{k}^{ANY},\quad &\mod(k,T)=2;\\
		\alpha_{k-1},\quad &\text{otherwise}.
	\end{aligned}
	\right.
\end{equation}

The algorithm framework of the extended NY method (ANY) method is summarized in Algorithm \ref{algANY}.		

\begin{algorithm}[h]
	\caption{ANY algorithm}
	\label{algANY}
	\rm 
	\textbf{Input:} Initial point $\mathbf{x}_0$; $k=0$; $T$; the stopping parameters $\epsilon$. \\
	\textbf{Output:} $\mathbf{x}_k$.	
	\begin{tabbing}
		\hspace{1.5em} \= \hspace{1.5em} \= \kill
		
		\textbf{while} $\|\mathbf{g}_k\| > \epsilon \|\mathbf{g}_0\|$ \textbf{do} \\
		
		\> \textbf{if} $\text{mod}(k, T) < 2$ \textbf{then} \\
		\> \> Let $\alpha_k = \alpha_k^{ASD}$, where $\alpha_k^{ASD}$ is computed from \cite[Algorithm 1]{ACINP}; \\
		
		\> \textbf{else if} $\text{mod}(k, T) = 2$ \textbf{then} \\
		\> \> Compute $\alpha_k = \alpha_k^{ANY}$;\\
		
		\> \textbf{else} \\
		\> \> Let $\alpha_k = \alpha_{k-1}$; \\
		\> \textbf{end if} \\
		
		\> Let $\alpha_k := \max\{\alpha_{\min}, \min\{\alpha_k, \alpha_{\max}\}\}$; \\
		\> Compute $\lambda_k$ by the improved GLL line search\cite{HGLL} with $\alpha_k$ as the first trial;\\
		\> Let $\mathbf{x}_k := \mathbf{x}_k - \lambda_k \mathbf{g}_k$ and $k := k + 1$. \\
		
		\textbf{end while}
	\end{tabbing}
\end{algorithm}
The improved GLL line search \cite{HGLL} used here is different from the traditional GLL line search \cite{GLL1986} in the backtracking strategy. Specifically, while the classic approach reduces the stepsize by a fixed factor, the improved line search adaptively updates the stepsize via quadratic interpolation. The detailed process is shown in the Appendix. As rigorously analyzed in \cite{ref5}, the convergence results established for the traditional GLL in \cite{GLL1986} is also valid for the improved version. Thus the following theoretical analysis directly follows the convergence results of the traditional GLL\cite{GLL1986}.

Next, the convergence properties of Algorithm \ref{algANY} are analyzed. The next theorem shows the global convergence of Algorithm \ref{algANY}.
\begin{theorem} \label{thm:global_convergence}
	Suppose that $\Omega = \{\mathbf{x} \in \mathbb{R}^n : f(\mathbf{x}) \leq f(\mathbf{x}_0)\}$ is bounded, and the objective function $f : \mathbb{R}^n \to \mathbb{R}$ of (\ref{f1.1}) is continuously differentiable in the neighborhood of $\Omega$. Let $\{\mathbf{x}_k\}$ be the sequence generated by Algorithm \ref{algANY} for problem (\ref{f1.1}). Then, either $\mathbf{g}_k = 0$ for some finite $k$, or the following properties hold:
	\begin{enumerate}
		\item[(a)] $\{\mathbf{x}_k\} \subset \Omega$, and $\nabla f(\bar{\mathbf{x}}) = 0$ holds for any limit point $\bar{\mathbf{x}}$ of $\{\mathbf{x}_k\}$;
		\item[(b)] no limit point of $\{\mathbf{x}_k\}$ is a local maximum of $f$;
		\item[(c)] if there are finite stationary points of $f$ in $\Omega$, then the sequence $\{\mathbf{x}_k\}$ converges to a stationary point of (\ref{f1.1}).
	\end{enumerate}
\end{theorem}

\begin{proof}
	Algorithm \ref{algANY} applies the improved GLL nonmonotone line search, and its theoretical results follows the traditional one. To share the global convergence result as GLL in \cite{GLL1986}, we need to prove that Algorithm \ref{algANY} satisfies the following three assumptions required by the theorem in \cite{GLL1986}, which are:	
	
	(i) there exists a constant $c_1 > 0$, such that:
	\begin{equation} \label{eq:cond_1}
		\mathbf{g}_k^T \mathbf{d}_k \leq -c_1 \Vert\mathbf{g}_k\Vert^2,
	\end{equation}
	where $\mathbf{d}_k$ is the search direction;
	
	(ii) there exists a constant $c_2 > 0$, such that:
	\begin{equation} \label{eq:cond_2}
		\Vert\mathbf{d}_k\Vert \leq c_2 \Vert\mathbf{g}_k\Vert;
	\end{equation}
	
	(iii) Let $\lambda_k = \rho^{p_k} \alpha_k$ be the stepsize corresponding to $\mathbf{d}_k$, where $\rho \in (0, 1)$ and $p_k$ is the first nonnegative integer $p$ satisfying:
	\begin{equation} \label{eq:cond_3}
		f(\mathbf{x}_k + \rho^p \alpha_k \mathbf{d}_k) \leq \max_{0 \leq j \leq m(k)} f(\mathbf{x}_{k-j}) + \delta \rho^p \alpha_k \mathbf{g}_k^T \mathbf{d}_k,
	\end{equation}
	where $m(0) = 0$ and $0 \leq m(k) \leq \min\{m(k-1) + 1, M\}$, $k \geq 1$, $\delta \in (0, 1)$, and $M \in \mathbb{N}$.
	
	It is obvious that the search direction $\mathbf{d}_k = -\mathbf{g}_k$ in Algorithm \ref{algANY} satisfies assumptions (i) and (ii) with $c_1 = c_2 = 1$. Since the algorithm employs the GLL nonmonotone line search in each iteration, assumption (iii) holds with the initial trial stepsize $a = \alpha_k$. As noted in the proof of the theorem in \cite{GLL1986}, it does not require $a$ to be a constant, but it must be positive. This condition is satisfied in Algorithm \ref{algANY} since the trial stepsize $\alpha_k$ has a positive lower bound $\alpha_{\min}$.
	
	Hence the conclusion of the theorem in \cite{GLL1986} holds. And Theorem \ref{thm:global_convergence} is proved.\hfill $\square$
\end{proof}
Next the convergence rate of the Algorithm \ref{algANY} is analyzed. 

\begin{theorem} \label{thm:convergence_rate}
	Suppose $\nabla f$ is continuous Lipschitz on $\mathbb{R}^n$, which means
	\begin{equation*}
		\|\nabla f(\mathbf{x}) - \nabla f(\mathbf{y})\| \le L \|\mathbf{x} - \mathbf{y}\|,\quad \forall \mathbf{x}, \mathbf{y} \in \mathbb{R}^n.
	\end{equation*}
	 Let $\{\mathbf{x}_k\}$ be a sequence generated by the Algorithm \ref{algANY} for problem (\ref{f1.1}), and assume $f$ attains the minimum at the point $\mathbf{x}^*$. Then, the following two cases hold.
	
	\begin{enumerate}
		\item[(a)] If $f$ is convex, then there exists a constant $C > 0$ such that for any $k$, it holds that
		\begin{equation}
			f(\mathbf{x}_k) - f(\mathbf{x}^*) \le \frac{C}{k}.
		\end{equation}
		
		\item[(b)] If $f$ satisfies the quadratic growth property\footnote{\cite{Huang2015} The continuously differentiable convex function $f$ has quadratic functional growth on the set $\Omega$ if there exists a positive constant $\eta$ satisfying the condition that for any $\mathbf{x}\in\Omega$ the following inequality holds:
			\begin{equation*}
				f(\mathbf{x})-f(\mathbf{x}^{*})\geq \eta\Vert\mathbf{x}-\mathbf{x}^*\Vert^{2},\ \ \ \forall\ \mathbf{x}\in\Omega,
			\end{equation*}
			where $\mathbf{x}^{*}$ is the projection point of $\mathbf{x}$ onto the optimal solution set of minimizing $f$.}, then there exist $\theta \in (0, 1)$ and $\nu > 0$ such that for any $k$, it holds that
		\begin{equation}
			\frac{f(\mathbf{x}_k) - f(\mathbf{x}^*)}{f(\mathbf{x}_0) - f(\mathbf{x}^*)} \le \nu \theta^k.
		\end{equation}
	\end{enumerate}
\end{theorem}

\begin{proof}
	The proof follows from the convergence rate analysis established in \cite[Theorem 4.1]{ref5}.
	In Algorithm \ref{algANY}, the stepsize $\alpha_k$ is restricted to the interval $[\alpha_{\min}, \alpha_{\max}]$. Combining with the GLL nonmonotone line search, we are able to get to the conclusions (a) and (b). \hfill $\square$
\end{proof}

\section{Numerical Tests}\label{section4}
In this section, we present numerical comparisons of our proposed new methods with the benchmark gradient methods on optimization problems with both quadratic and general nonquadratic objective functions. All tests are done on MATLAB 2021a, Intel Core i9-13900.

For the new methods proposed in this paper, the parameter settings are as follows:
\begin{equation*}
	\alpha_{\min}=10^{-10}, \quad \alpha_{\max}=10^{5}, \quad T=7, \quad \epsilon=10^{-6}
\end{equation*}
and the parameters of the adaptive nonmonotone line search are the same as \cite{HGLL}. 
\subsection{Algorithm comparison}

In this part, a total of $93$ unconstrained test problems from CUTEst\cite{CUTEst} with problem dimensions from $4$ to $20,000$ are considered. The following gradient methods are tested: 
\begin{itemize}
	\item cyclic gradient methods with SL framework: AC-INP-YV \cite{ACINP}, ZS-YV\cite{Zhangya2022};
	\item BB-type methods: ABBmin\cite{ABBmin19}, MPSG\cite{Huang2022};
	\item conjugate gradient methods: CGDY\cite{CGDY}, CGOPT\cite{CGOPT}.
\end{itemize}
The parameters in the compared methods are the same as in the corresponding references.
For all the tested methods, the maximum iteration number is set as $20,000$. If the stopping condition for the method is not attained within $20,000$ iterations, or it encounters ``Inf'' or ``NaN'', the method is treated as unsolved. 
The performance profile \cite{Dolan2002} is used here for the algorithm comparison, with different metrics. The vertical axis illustrates the proportion of problems that the method solves within a factor $\tau$ of the most effective method's metric, where $\log_2\tau$ is its horizontal axis.

\begin{figure}[!htb]
	\centering
	\includegraphics[width=\textwidth]{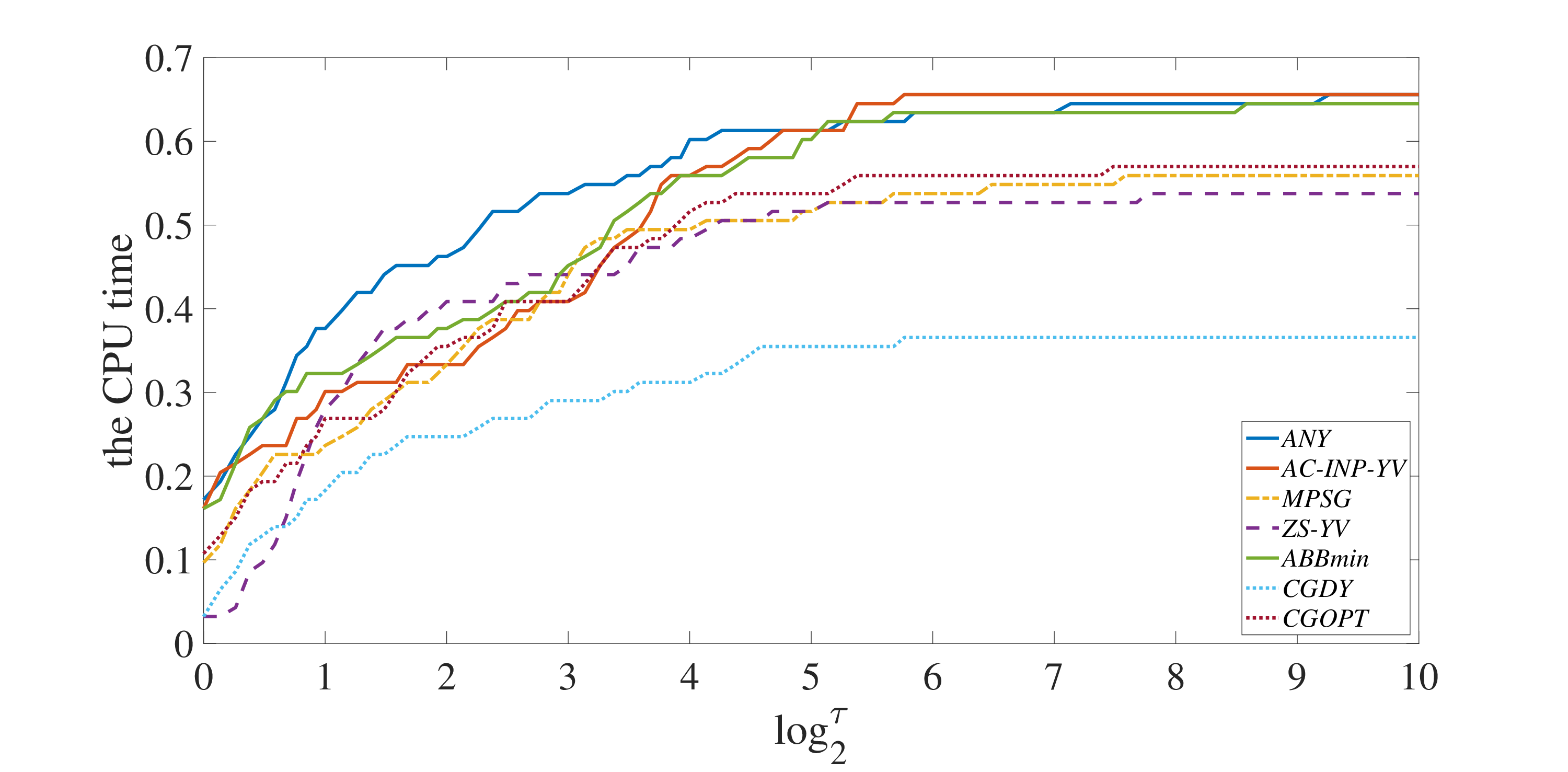}%
	\vskip -1mm
	\caption{Performance profile on the CPU time metric}
	\label{ANY_t}
\end{figure}
\begin{figure}[!htb]
	\centering
	\includegraphics[width=\textwidth]{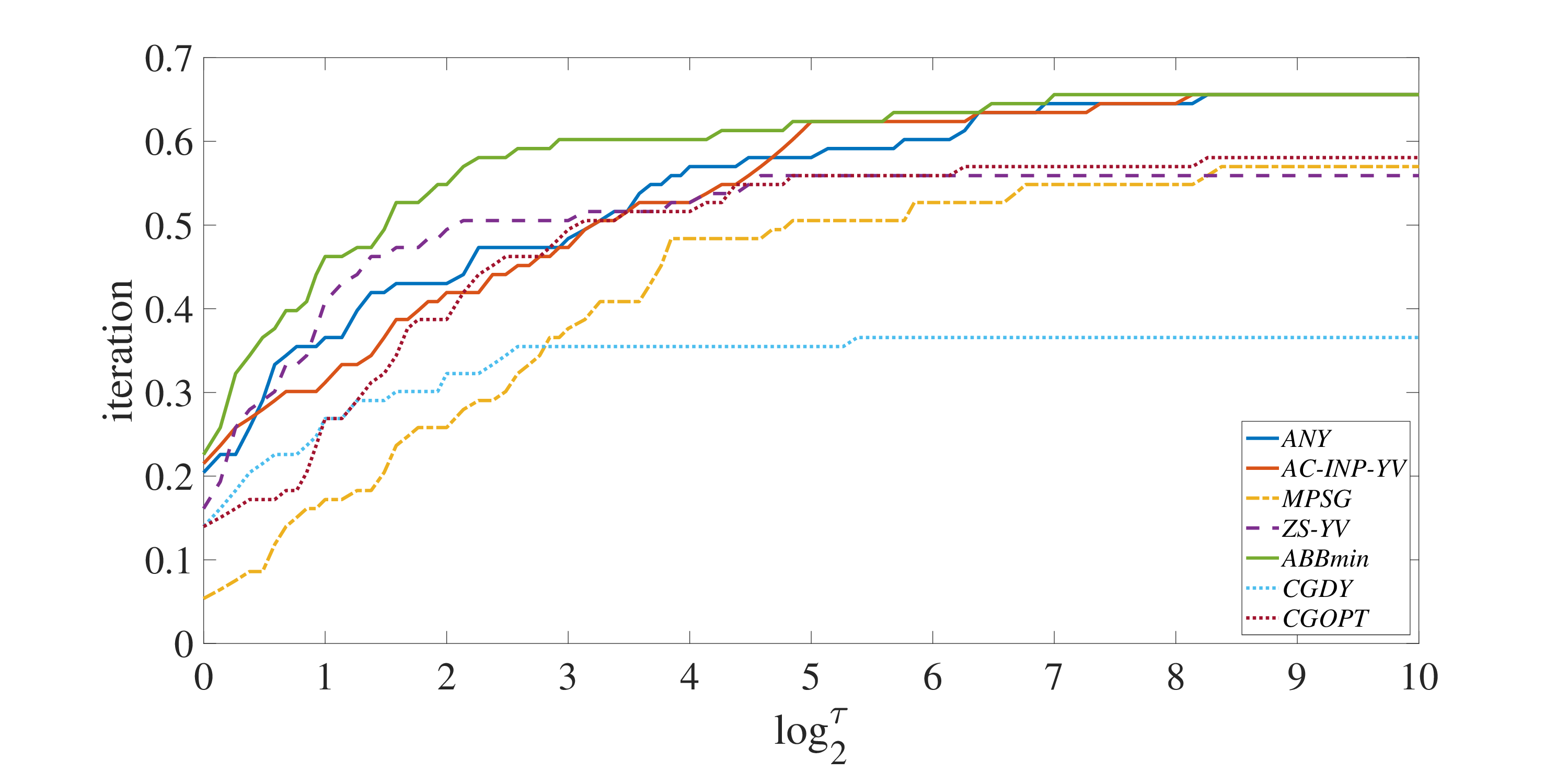}%
	\vskip -1mm
	\caption{Performance profile on the iteration metric}
	\label{ANY_k}
\end{figure}

The performance profiles for all the compared methods regarding the CPU time and the number of iterations are plotted in Fig. \ref{ANY_t} and Fig \ref{ANY_k}, respectively. 
The proposed method is able to solve about $68\%$ test problems, which is the best among all compared methods. Consider the consumed CPU time, and we can see that the proposed method performs the best while ABBmin requires the least iteration number. This shows that our proposed method enjoys less computational cost averagely in each iteration compared to the benchmark methods.
\begin{figure}[!htb]
	\centering
	\includegraphics[width=\textwidth]{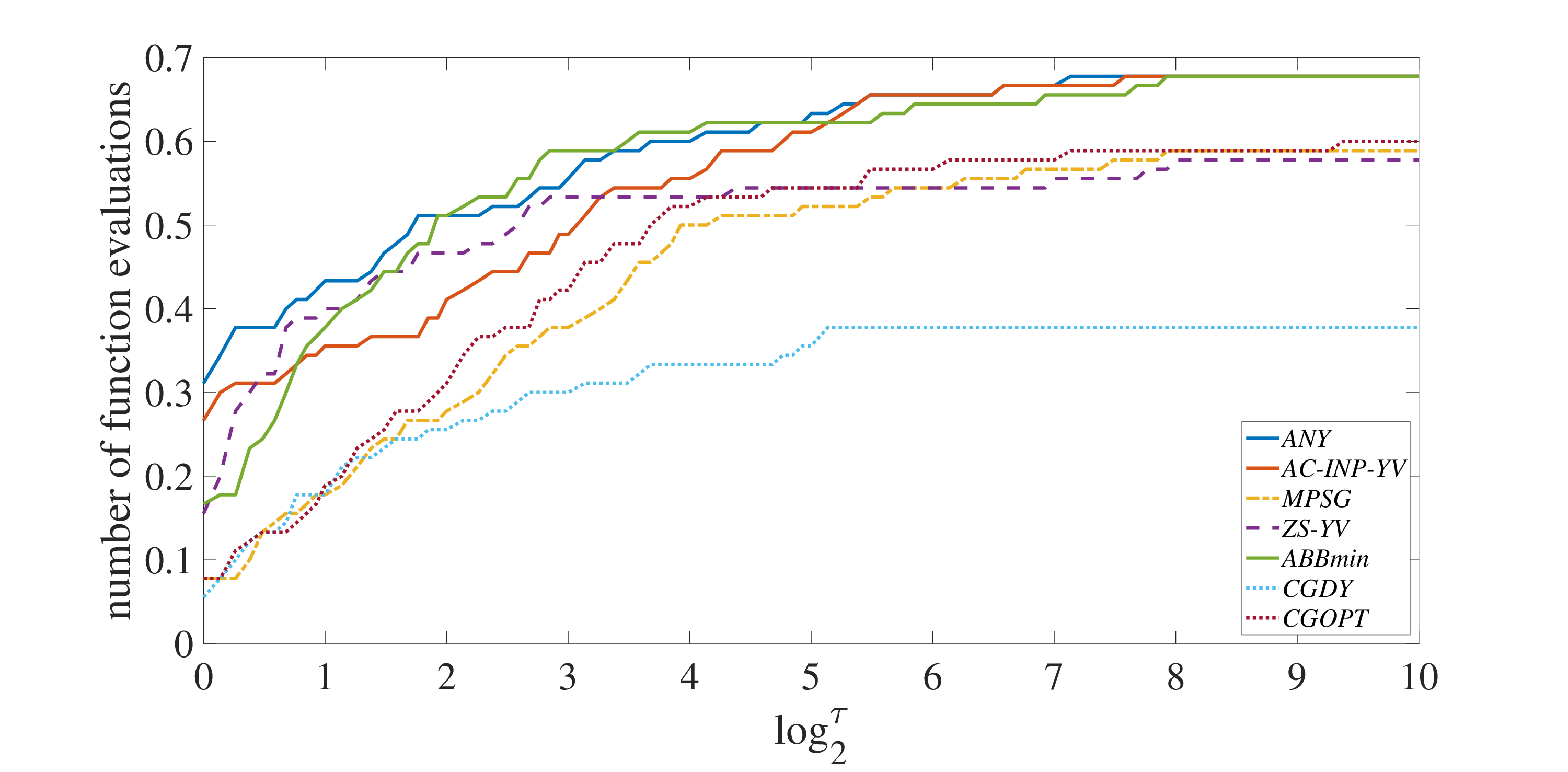}%
	\vskip -1mm
	\caption{Performance profile on the number of function value metric}
	\label{ANY_fi}
\end{figure}
\begin{figure}[!htb]
	\centering
	\includegraphics[width=\textwidth]{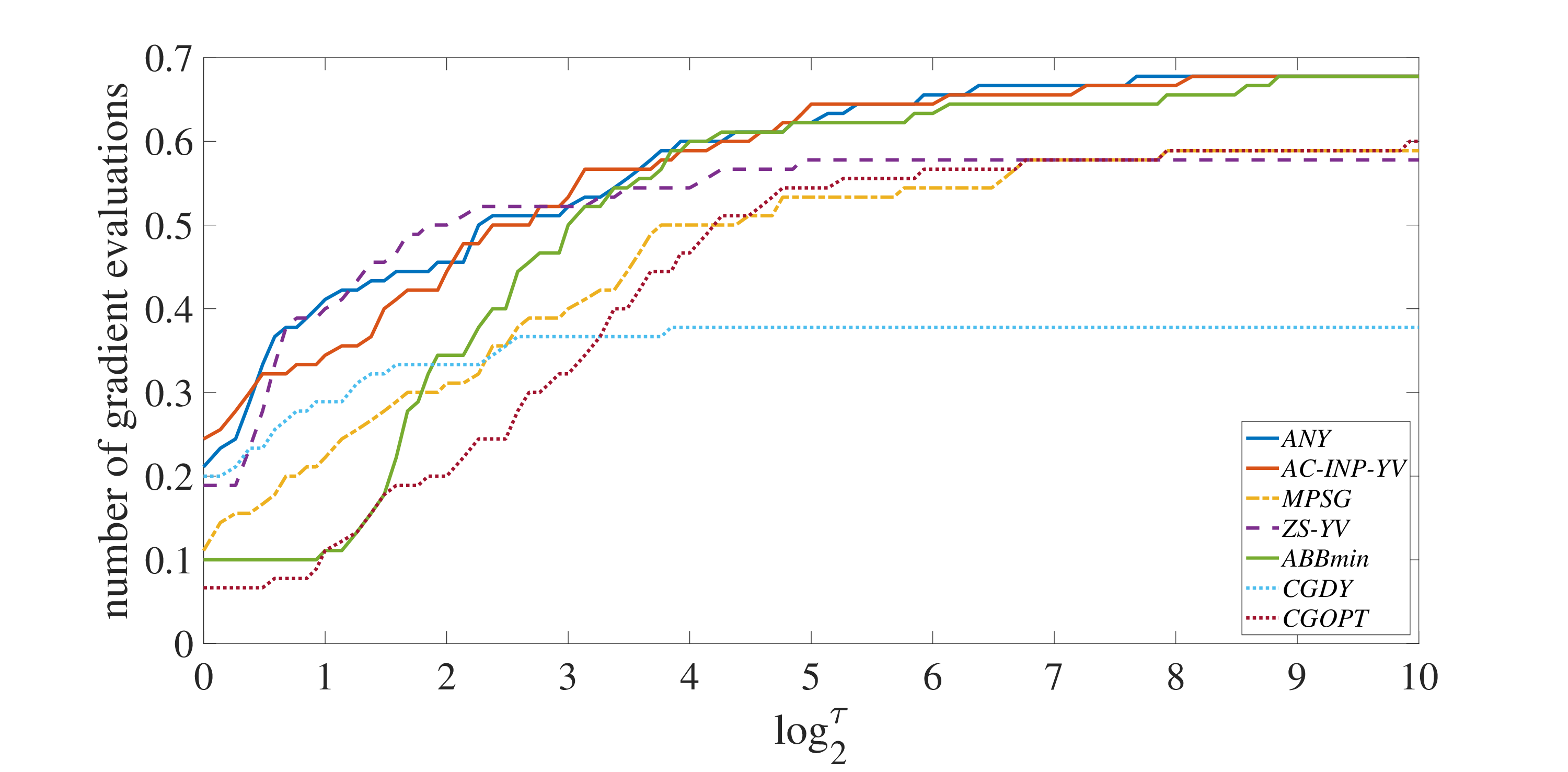}%
	\vskip -1mm
	\caption{Performance profile on the number of gradient value metric}
	\label{ANY_gi}
\end{figure}

The performance profiles for the number of function value and gradient evaluations are depicted in Fig. \ref{ANY_fi} and \ref{ANY_gi}, respectively. From them we can also observe the efficiency of the cyclic framework of the proposed method, which requires almost the least number of function value and gradient evaluations. 

Good performances of the proposed method are observed through the experiments. In the next, we will further analyze the proposed method, and show the possible reason for its good performance.

\subsection{Line search analysis}
For the optimization of general problem (\ref{f1.1}), it usually requires line search technique, so that the convergence of the method is guaranteed. We use backtracking to satisfy the line search requirement. However, frequent backtracking causes additional function evaluations, which are often the computational bottleneck in large scale optimization. In this part, the line search efficiency of the compared methods are analyzed.

We introduce the metric ``average line search counts per iteration'', defined as the average number of additional function evaluations required per iteration for the line search condition satisfaction. Less number of backtracking indicates that the trial stepsizes provided by the algorithm are effective, which satisfy the line search condition easily.

Here we consider the following methods which perform well in the previous tests, and their corresponding line search techniques as in the references:
\begin{itemize}
	\item \textbf{ANY}: the proposed method with the improved GLL line search \cite{HGLL};
	\item \textbf{AC-INP-YV\cite{ACINP}}: the cyclic gradient method based on interpolation, with the Dai-Zhang adaptive nonmonotone line search\cite{Dai2001ATSG};
	\item \textbf{ZS-YV\cite{Zhangya2022}}: the cyclic gradient method based on step-ahead BB approximation, with the GLL nonmonotone line search \cite{GLL1986};
	\item \textbf{ABBmin\cite{ABBmin19}}: the improved BB gradient method with the GLL nonmonotone line search \cite{GLL1986}.
\end{itemize}
Among all the test problems, 40 of them can be solved successfully by all the aforementioned four algorithms. The average number of extra line search trial stepsizes per iteration of the four methods for each test problem is plotted in Fig. \ref{fig:ls_comparison}. The logarithm vertical axis is adopted, so that the comparison is clear. The red, blue, green and yellow bars represent the four methods respectively. The vanished bar in some test problems means that there is no extra line search trial stepsize for the corresponding method (initial trial stepsize is accepted immediately). 
\begin{figure}[!htb]
	\centering
	\includegraphics[width=\textwidth]{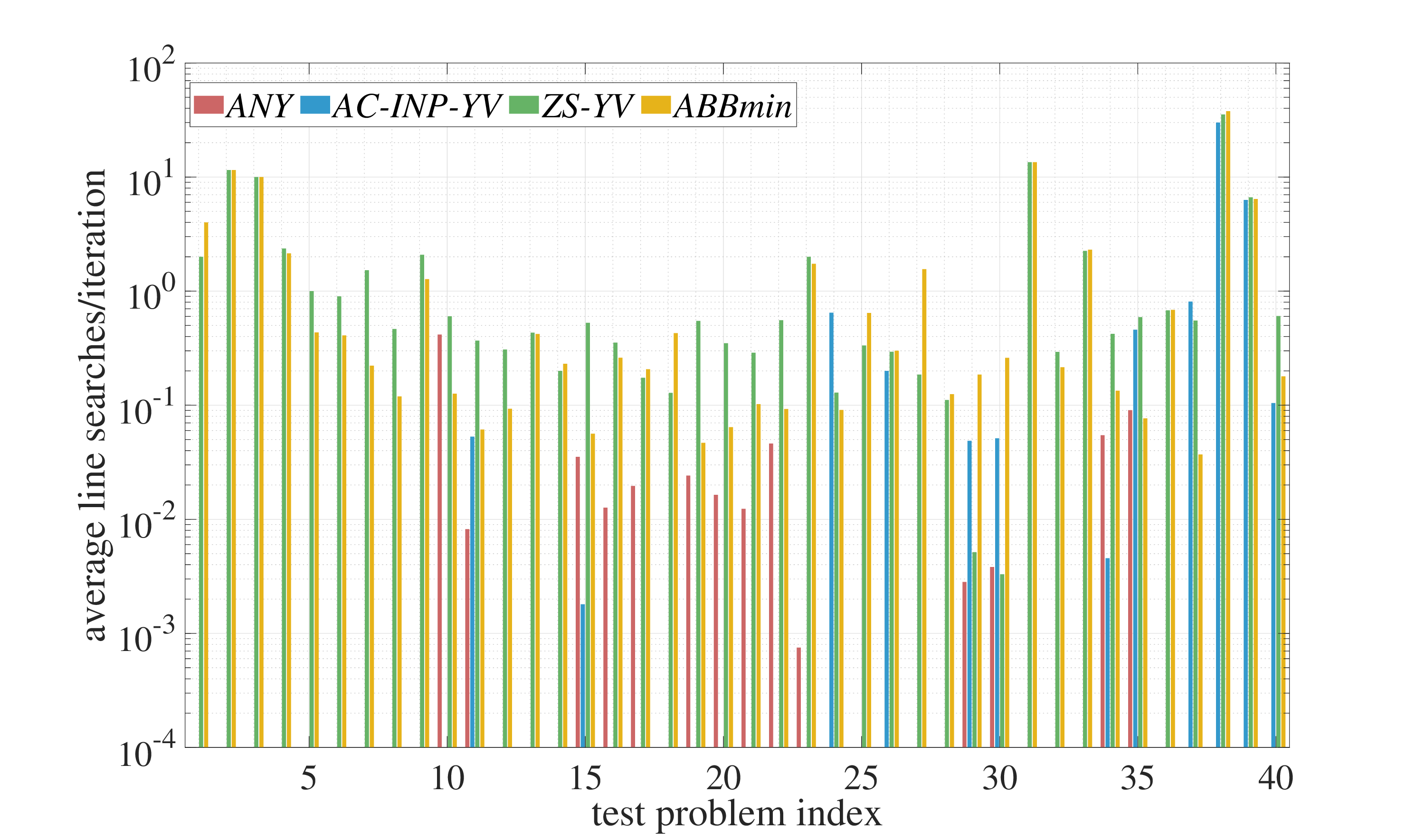}%
	\vskip -1mm
	\caption{Comparison of average line search counts per iteration.}
	\label{fig:ls_comparison}
\end{figure}

As observed in Figure \ref{fig:ls_comparison}, the red bars are usually shorter than the other bars. This shows that the proposed ANY method provides effective stepsizes, which require few line search trial stepsizes.

In detail, the average line search trial stepsize per iteration for each test problem is only $0.0186$, considering the proposed ANY method. This value becomes $0.9669$ for AC-INP-YV,  $2.5146$ for ZS-YV and $2.4632$ for ABBmin. Statistical results imply that the initial trial stepsize generated by the ANY method is accepted immediately in more than $98\%$ of the iterations. This provides strong evidence that the proposed method works well and requires less computational cost than the benchmark methods.

\begin{remark}
	Here the compared methods adopt different line searches, which keep the same as in their references. Usually the applied line search technique is proper for the corresponding method. For example, ABBmin uses the GLL line search\cite{GLL1986}. If we replace it by the improved GLL\cite{HGLL}, the  average line search trial stepsize per iteration for each test problem would become $6.8620$, which is worse than ABBmin method with traditional GLL.
\end{remark}

Through the analysis, it is observed that the proposed stepsize update strategy is effective. Most stepsizes can satisfy the line search condition naturally.

\subsection{High dimensional tests}

In this subsection, the test on high dimensional problems is shown. 

$9$ test problems are considered, with dimensions $10^5$ and $10^6$. Problems $1$-$3$ are strictly convex quadratic functions in the form of \eqref{quadraticf} with $\kappa=10^6$ and $\boldsymbol{H} = \text{diag}(\lambda_1, \cdots, \lambda_n)$. Problems $4-9$ are general unconstrained problem. Their detailed formulas are as follows:

\begin{itemize}
	\item \textbf{Problem 1 \cite{Zhou2006}:} $\lambda_1 = 0.1$ and $\lambda_i = i$ for $i > 1$, with $\boldsymbol{b} = (1,\dots,1)^T$ and $\boldsymbol{x}_0 = \mathbf{0}$.
	
	\item \textbf{Problem 2 \cite{Sun2020}:} $\lambda_i\in[1, 1 + 0.2(\kappa-1)], i=1, \cdots, \frac{n}{2}$ and $\lambda_i\in[0.8\kappa, \kappa]i=\frac{n}{2}+1, \cdots, n$, with $\boldsymbol{b} = \mathbf{0}$ and choose $\boldsymbol{x}_0$ on the unit sphere.
	
	\item \textbf{Problem 3 \cite{Sun2020}:} $\lambda_i = \frac{\kappa}{2} [\cos(\frac{n-i}{n-1}\pi) + 1]$ with $\boldsymbol{b} = \mathbf{0}$ and choose $\boldsymbol{x}_0$ on the unit sphere.
	
	\item \textbf{Problem 4 (BROYDN3D \cite{CUTEst}):} 
	\[ f(\boldsymbol{x}) = \sum_{i=1}^n \left( (3 - 2\boldsymbol{x}_i)\boldsymbol{x}_i - \boldsymbol{x}_{i-1} - 2\boldsymbol{x}_{i+1} + 1 \right)^2, \]
	where $\boldsymbol{x}_0 = (-1, \cdots, -1)^T$.
	
	\item \textbf{Problem 5 (COSINE \cite{CUTEst}):} 
	\[ f(\boldsymbol{x}) = \sum_{i=1}^{n-1} \cos(\boldsymbol{x}_i^2 - 0.5\boldsymbol{x}_{i+1}). \]
	where $\boldsymbol{x}_0 = (1, \cdots, 1)^T$.
	
	\item \textbf{Problem 6 (DIXMAANJ \cite{CUTEst}):} 
	\[ f(\boldsymbol{x}) = 1 + \alpha \sum_{i=1}^n \left(\tfrac{i}{n}\right)^{2} \boldsymbol{x}_i^2 + \beta \sum_{i=1}^{n-1} \boldsymbol{x}_i^2 (\boldsymbol{x}_{i+1} + \boldsymbol{x}_{i+1}^2)^2 + \gamma \sum_{i=1}^{2m} \boldsymbol{x}_i^2 \boldsymbol{x}_{i+m}^4 + \delta \sum_{i=1}^m \left(\tfrac{i}{n}\right)^{2} \boldsymbol{x}_i \boldsymbol{x}_{i+2m}, \]
	where $\alpha=1, \beta=\gamma=\delta=0.0625$ and $\boldsymbol{x}_0 = (2, \cdots, 2)^T$.
	
	\item \textbf{Problem 7 (ENGVAL1 \cite{CUTEst}):} 
	\[ f(\boldsymbol{x}) = \sum_{i=1}^{n-1} \left[ (\boldsymbol{x}_i^2 + \boldsymbol{x}_{i+1}^2)^2 - 4\boldsymbol{x}_i + 3 \right], \]
	where $\boldsymbol{x}_0 = (2, \cdots, 2)^T$.
	
	\item \textbf{Problem 8 (FIROSE \cite{Li1988}):} The objective is a sum of squares $f(\boldsymbol{x}) = \sum_{i=1}^n F_i^2(\boldsymbol{x})$, where the dominant term is:
	\[ F_i = 8\boldsymbol{x}_i(\boldsymbol{x}_i^2 - \boldsymbol{x}_{i-1}) - 2(1-\boldsymbol{x}_i) + 4(\boldsymbol{x}_i - \boldsymbol{x}_{i+1}^2) + \boldsymbol{x}_{i-1}^2 - \boldsymbol{x}_{i-2} + \boldsymbol{x}_{i+1} - \boldsymbol{x}_{i+2}^2, \]
	where $\boldsymbol{x}_0 = (-1, \cdots, -1)^T$.
	
	\item \textbf{Problem 9 (TRIROSE2 \cite{Li1988}):} 
	\[ 
	\begin{aligned}
		f(\boldsymbol{x}) = & \, 16(\boldsymbol{x}_1 - \boldsymbol{x}_2^2)^2 + \sum_{i=2}^{n-1} \left[ 8\boldsymbol{x}_i(\boldsymbol{x}_i^2 - \boldsymbol{x}_{i-1}) - 2(1-\boldsymbol{x}_i) + 4(\boldsymbol{x}_i - \boldsymbol{x}_{i+1}^2) \right]^2 \\
		& + \left[ 8\boldsymbol{x}_n(\boldsymbol{x}_n^2 - \boldsymbol{x}_{n-1}) - 2(1-\boldsymbol{x}_n) \right]^2,
	\end{aligned}
	\]
	where $\boldsymbol{x}_0 = (-1, \cdots, -1)^T$.
	
\end{itemize}
The following methods are tested and compared:
\begin{itemize}
	\item[1.] proposed: Algorithm \ref{algNY} for quadratic case, Algorithm \ref{algANY} for nonquadratic case;
	\item[2.] BB type: ABBmin\cite{ABBmin19}, MPSG\cite{Huang2022};
	\item[3.] SL type cyclic gradient methods: SL method\cite{Sun2020} for quadratic case, ZS-YV method\cite{Zhangya2022} and AC-INP-YV method\cite{ACINP} for nonquadratic case.
\end{itemize}
The computation results are summarized in Table \ref{tab:high_dim_comparison}. The column ``Iter'' denotes the number of iterations required by the corresponding algorithm for the certain problem, and ``Time'' represents CPU time in seconds. The symbol ``-'' indicates that the algorithm failed to converge within $20,000$ iterations, or that it terminates with abnormal situations like encountering ``Inf'' or ``NaN'' values.

Generally the proposed methods have best performances both in computational time and iteration number. It solves all the tested problems successfully. Especially for test problem 1, the proposed NY method solves it successfully while all the other compared methods fail. The sharp comparison is due to the high condition number ($10^6$ and $10^7$) of the problem. This indicates that our proposed methods not only work well for high dimensional problems, but also has the ability to deal with ill-conditioned problems.

\begin{table}[htbp]
	\centering
	\small 
	\setlength{\tabcolsep}{3pt} 
	
	\caption{high dimensional problem tests}
	\label{tab:high_dim_comparison}
	
	\begin{tabular}{|c|c|c|c|c|c|c|c|}
		\hline
		\textbf{Prob.} & \textbf{Dim.} & \textbf{Metric} & \textbf{ANY} & \textbf{ABBmin} & \textbf{MPSG} & \textbf{ZS-YV} & \textbf{AC-INP-YV} \\ \noalign{\hrule height 1.2pt}
		
		1 & $10^5$ & Iter & \textbf{8838} & - & - & \multicolumn{2}{c|}{-} \\ \cline{3-8} 
		&        & Time & \textbf{7.07} & - & - & \multicolumn{2}{c|}{-} \\ \cline{2-8} 
		& $10^6$ & Iter & \textbf{13199} & - & - & \multicolumn{2}{c|}{-} \\ \cline{3-8} 
		&        & Time & \textbf{65.30} & - & - & \multicolumn{2}{c|}{-} \\ \noalign{\hrule height 1.2pt}
		
		2 & $10^5$ & Iter & 22 & \textbf{19} & 20 & \multicolumn{2}{c|}{25} \\ \cline{3-8} 
		&        & Time & 0.013 & \textbf{0.012} & \textbf{0.012} & \multicolumn{2}{c|}{0.014} \\ \cline{2-8} 
		& $10^6$ & Iter & 22 & \textbf{19} & 20 & \multicolumn{2}{c|}{25.4} \\ \cline{3-8} 
		&        & Time & \textbf{0.133} & 0.155 & 0.446 & \multicolumn{2}{c|}{0.160} \\ \noalign{\hrule height 1.2pt}
		
		3 & $10^5$ & Iter & \textbf{229} & 257.4 & 755.2 & \multicolumn{2}{c|}{308} \\ \cline{3-8} 
		&        & Time & \textbf{0.041} & 0.044 & 0.045 & \multicolumn{2}{c|}{0.188} \\ \cline{2-8} 
		& $10^6$ & Iter & \textbf{225} & 258.2 & 825.2 & \multicolumn{2}{c|}{345.2} \\ \cline{3-8} 
		&        & Time & \textbf{1.14} & 1.86 & 91.89 & \multicolumn{2}{c|}{2.16} \\ \noalign{\hrule height 1.2pt}
		
		4 & $10^5$ & Iter & \textbf{24} & - & - & 79 & - \\ \cline{3-8} 
		&        & Time & \textbf{0.09} & - & - & 0.50 & - \\ \cline{2-8} 
		& $10^6$ & Iter & \textbf{21} & 1844 & 186 & 74 & - \\ \cline{3-8} 
		&        & Time & \textbf{0.46} & 289 & 5.02 & 3.36 & - \\ \noalign{\hrule height 1.2pt}
		
		5 & $10^5$ & Iter & 21 & 54 & - & - & \textbf{14} \\ \cline{3-8} 
		&        & Time & \textbf{0.12} & 8.17 & - & - & 0.17 \\ \cline{2-8} 
		& $10^6$ & Iter & 20 & 256 & - & - & \textbf{17} \\ \cline{3-8} 
		&        & Time & \textbf{0.86} & 858 & - & - & 1.69 \\ \noalign{\hrule height 1.2pt}
		
		6 & $10^5$ & Iter & \textbf{66} & 406 & 11558 & 69 & 73 \\ \cline{3-8} 
		&        & Time & \textbf{1.56} & 15.1 & 229 & 2.97 & 1.64 \\ \cline{2-8} 
		& $10^6$ & Iter & \textbf{66} & 356 & 4109 & 118 & 117 \\ \cline{3-8} 
		&        & Time & \textbf{15.5} & 128 & 687 & 47.5 & 25.0 \\ \noalign{\hrule height 1.2pt}
		
		7 & $10^5$ & Iter & 28 & \textbf{12} & 93 & 19 & 29 \\ \cline{3-8} 
		&        & Time & 0.08 & \textbf{0.05} & 0.26 & 0.08 & 0.06 \\ \cline{2-8} 
		& $10^6$ & Iter & 24 & \textbf{12} & 99 & 17 & 19 \\ \cline{3-8} 
		&        & Time & 0.51 & \textbf{0.47} & 2.13 & 0.55 & 0.77 \\ \noalign{\hrule height 1.2pt}
		
		8 & $10^5$ & Iter & 174 & 250 & - & \textbf{102} & - \\ \cline{3-8} 
		&        & Time & \textbf{0.71} & 2.00 & - & 0.91 & - \\ \cline{2-8} 
		& $10^6$ & Iter & \textbf{116} & 3746 & - & - & - \\ \cline{3-8} 
		&        & Time & \textbf{3.86} & 801 & - & - & - \\ \noalign{\hrule height 1.2pt}
		
		9 & $10^5$ & Iter & \textbf{137} & 234 & - & - & 302 \\ \cline{3-8} 
		&        & Time & \textbf{0.46} & 1.28 & - & - & 0.88 \\ \cline{2-8} 
		& $10^6$ & Iter & \textbf{93} & - & 175 & - & 218 \\ \cline{3-8} 
		&        & Time & \textbf{2.24} & - & 4.76 & - & 5.09 \\ \noalign{\hrule height 1.2pt}
	\end{tabular}
\end{table}
\newpage		
\section{Conclusion}\label{section5}
A new type of stepsize for gradient method is proposed, which is deduced through the aim of $5$ step termination for $3$ dimensional quadratic function minimization problems. The proposed NY stepsize is combined with the cyclic update strategy, and a new gradient method is proposed. It is proved to have $3$ dimensional quadratic termination property. The proposed gradient method is extended for solving general unconstrained problem, through the quadratic interpolation approximation for Cauchy step. By integrating the improved GLL nonmonotone line search, we \nobreak establish the global convergence of the proposed method. Furthermore, we analyze its sublinear \nobreak convergence rate for convex problems and $R$-linear convergence rate for problems with quadratic \nobreak functional growth property. Numerical results show that our proposed algorithm has the best performances in terms of the computational time compared to the benchmark methods. The proposed method also provides good initial trials for line search, which leads to little computational cost for line search.

	\appendix
	\section{The Improved GLL Line Search}
	\label{app:linesearch}
	
	In this appendix, the detailed implementation of the improved GLL line search \cite{HGLL} used in the ANY algorithm (Algorithm \ref{algANY}) is provided. It integrates the GLL nonmonotone line search with quadratic interpolation backtracking. Unlike the traditional approach, the new trial stepsize is computed using the local information of the current iteration to capture the local curvature accurately. The brief algorithm framework is summarized in Algorithm \ref{alg:improved_gll}.
	
\begin{algorithm}[H]
	\caption{The improved GLL line search}
	\label{alg:improved_gll}
	\textbf{Input:} Current iterate $\mathbf{x}_k$, gradient $\mathbf{g}_k$, function value $f(\mathbf{x}_k)$, initial stepsize $\alpha$, history length $M$, parameter $\delta$, max line search iterations $L_{\max}$  \\
	\textbf{Output:} Accepted stepsize $\alpha_k$ and new iterate $\mathbf{x}_{k+1}$.
	\begin{tabbing}
		\hspace{1.5em} \= \hspace{1.5em} \= \hspace{1.5em} \= \kill
		Let $f_{\text{ref}} = \max_{0 \le j \le \min\{k, M\}} f(\mathbf{x}_{k-j})$; \\
		Let $iter_{ls} = 0$;\\
		\textbf{while} $iter_{ls} < L_{max}$ \textbf{do} \\
		\> $iter_{ls} = iter_{ls} + 1$;\\
		\> Let $\mathbf{x}_{\text{new}} = \mathbf{x}_k - \alpha \mathbf{g}_k$ and $f_{\text{new}} = f(\mathbf{x}_{\text{new}})$; \\
		
		\> \textbf{if} $f_{\text{new}} \le f_{\text{ref}} - \delta \alpha \Vert\mathbf{g}_k\Vert^2$ \textbf{then} \\
		\> \> Let $\alpha_k = \alpha, \mathbf{x}_{k+1} = \mathbf{x}_{\text{new}}$; \\
		\> \> \textbf{return} \\
		
		\> \textbf{else} \\
		\> \> Compute $\bar{\alpha} = \frac{\Vert\mathbf{g}_k\Vert^2\alpha^2}{2 [f(\mathbf{x}_k - \alpha \mathbf{g}_k) - f(\mathbf{x}_k) + \alpha \Vert\mathbf{g}_k\Vert^2]}$; \\
		
		\> \> \textbf{if} $\bar{\alpha} \in [0.1\alpha, 0.9\alpha]$ \textbf{then} \\
		\> \> \> Let $\alpha = \bar{\alpha}$; \\
		\> \> \textbf{else} \\
		\> \> \> Let $\alpha = 0.5\alpha$; \\
		\> \> \textbf{end if} \\
		\> \textbf{end if} \\
		\textbf{end while}
	\end{tabbing}
\end{algorithm}
	

\end{document}